\documentclass[12pt]{amsart}
\usepackage{graphicx}
\usepackage[headings]{fullpage}
\usepackage{amssymb,epic,eepic,epsfig,amsbsy,amsmath,amscd}
\numberwithin{equation}{section}
                        \theoremstyle{plain}
\usepackage{mathrsfs}

\newcommand\no[1]{}

\newtheorem{theorem}{Theorem}[section]
\newtheorem{thm}{Theorem}
\newtheorem{lemma}[theorem]{Lemma}

\newtheorem{proposition}[theorem]{Proposition}

\theoremstyle{definition}
\newtheorem{remark}[theorem]{Remark}
\newtheorem{example}[theorem]{Example}

\def\BC{\mathbb C}

\def\BZ{\mathbb Z}
\def\BR{\mathbb R}

\def\fb{\mathfrak b}

\def\la{\langle}
\def\ra{\rangle}

\DeclareMathOperator{\tr}{\mathrm tr}

\def\be { \begin{equation} }
\def\ee { \end{equation} }

\begin{document}

\title[left-orderability and cyclic branched coverings]{On left-orderability and cyclic branched coverings}

\author[Anh T. Tran]{Anh T. Tran}
\address{Department of Mathematics, The Ohio State University, Columbus, OH 43210, USA}
\email{tran.350@osu.edu}


\begin{abstract}
In a recent paper Y. Hu has given a sufficient condition for the fundamental group of the $r$-th cyclic branched covering of $S^3$ along a prime knot to be left-orderable in terms of representations of the knot group. Applying her criterion to a large class of two-bridge knots, we determine a range of the integer $r>1$ for which the $r$-th cyclic branched covering of $S^3$ along the knot is left-orderable.
\end{abstract}

\maketitle

\section{Introduction}

A non-trivial group $G$ is called left-orderable if there exists a strict total ordering $<$ on its elements such that $g<h$ implies $fg<fh$ for all elements $f,g,h \in G$. Knot groups and more generally the fundamental group of an irreducible 3-manifold with positive first Betti number are examples of left-orderable groups \cite{HSt}. Left-orderable groups have recently attracted the attention of many people partly because of its possible connection to L-spaces, a class of rational homology 3-spheres defined by Osvath and Szabo \cite{OS} using Heegaard Floer homology, via a conjecture of Boyer, Gordon and Watson \cite{BGW}. This conjecture predicts that an irreducible rational homology 3-sphere is an
L-space if and only if its fundamental group is not left-orderable. The conjecture has been confirmed for Seifert fibered manifolds, Sol manifolds, double branched coverings of non-splitting alternating links \cite{BGW} and Dehn surgeries on the figure eight knot, on the knot $5_2$ and more generally on genus one two-bridge knots (see \cite{BGW, CLW}, \cite{HT1} and \cite{HT2, HT3, Tr} respectively). A technique that has so far worked very well for proving the left-orderability of fundamental groups is lifting a non-abelian $SU(1,1)$ representation (or equivalently a non-abelian $SL_2(\BR)$ representation) of a 3-manifold group to the universal covering group $\widetilde{SU(1,1)}$ and then using the result by Bergman \cite{Be} that $\widetilde{SU(1,1)}$ is a left-orderable group. This technique, which is based on an important result of Khoi \cite{Kh}, was first introduced in \cite{BGW} and was applied in \cite{HT1, HT2, HT3, Tr} to study the left-orderability of Dehn surgeries on genus one two-bridge knots.

The left-orderability of the fundamental groups of non-hyperbolic rational homology 3-spheres has already been characterized in \cite{BRW}. For hyperbolic rational homology 3-spheres, many of them can be constructed from the cyclic branched coverings of $S^3$ along a knot.  Based on the Lin's presentation \cite{Li} of a knot group and the technique for proving the left-orderability of fundamental groups mentioned above, Y. Hu \cite{Hu} has recently given a suffcient condition for the fundamental group of the $r$-th cyclic branched covering of $S^3$ along a prime knot to be left-orderable in terms of representations of the knot group. As an application, she proves that for any two-bridge knot $\fb(p,m)$, with $p\equiv 3 \pmod{4}$, there are only finitely many cyclic branched coverings whose fundamental groups are not left-orderable. In particular for the two-bridge knots $5_2$ and $7_4$, Y. Hu shows that the fundamental groups of the $r$-th cyclic branched coverings of $S^3$ along them are left-orderable if $r \ge 9$ and $r \ge 13$ respectively.

In this paper by applying Hu's criterion to a large class of two-bridge knots, which includes the knots $5_2$ and $7_4$, we determine a range of the integer $r>1$ for which the $r$-th cyclic branched covering of $S^3$ along the knot is left-orderable. 

Let $K=J(k,l)$ be the double twist knot as in Figure 1. 
Note that $J(k,l)$ is a knot if and only if $kl$ is even, and is the trivial knot if $kl=0$. 
Furthermore, $J(k,l)\cong J(l,k)$ and 
$J(-k,-l)$ is the mirror image of $J(k,l)$. Hence, in the following, we consider $K=J(k,2n)$ for $k>0$ and $|n|>0$. 

In the Schubert's normal form $\fb(p,m)$, where $p, m$ are positive integers such that $p$ is odd and $0<m<p$, of a two-bridge knot one has $J(k,2n)=\fb(2kn-1,2n)$ if $n>0$ and $J(k,2n)=\fb(1-2kn,-2n)$ if $n<0$, see e.g. \cite{BZ}. 

For a knot $K$ in $S^3$ and any integer $r>1$, let $X_K^{(r)}$ be the $r$-th cyclic branched covering of $S^{3}$ along $K$. The following theorem generalizes Example 4.4 in \cite{Hu}.

\begin{thm}
\label{main}
Suppose $m$ and $n$ are positive integers. Then the group $\pi_1(X^{(r)}_{J(2m,2n)})$ is left-orderable if $r > \frac{\pi}{\cos^{-1}\sqrt{1-(4mn)^{-1}}}$.
\end{thm}

\begin{example} 

1) For the knot $5_2=J(4,2)$, the manifold $X^{(r)}_{5_2}$ has left-orderable fundamental group if $r > \frac{\pi}{\cos^{-1}\sqrt{\frac{7}{8}}} \approx 8.69$, i.e. $r \ge 9$.

2) For the knot $7_4=J(4,4)$, the manifold $X^{(r)}_{7_4}$ has left-orderable fundamental group if $r >  \frac{\pi}{\cos^{-1}\sqrt{\frac{15}{16}}} \approx12.43$, i.e. $r \ge 13$.
\end{example} 

\begin{remark}
\label{r1}
Da̧bkowski, Przytycki and Togha \cite{DPT} proved that the group $\pi_1(X^{(r)}_{J(2m,-2n)})$, for positive integers $m$ and $n$, is not left-orderable for any integer $r > 1$.
\end{remark}

 We also prove the following result in this paper.

\begin{thm} 
\label{main2}
Suppose $m \ge 0$ and $n>0$ are integers. Let $q=2n^2 +2n\sqrt{4m(m+1)+n^2}$. 

(a) The group $\pi_1(X^{(r)}_{J(2m+1, 2n)})$ is left-orderable if one of the following holds: 

\qquad (i) $n$ is even and $r>\frac{\pi} { \cos^{-1} \sqrt{ 1-q^{-1}} }$.

\qquad (ii) $n$ is odd $>1$ and $r>\max \{ \frac{\pi} { \cos^{-1} \sqrt{ 1-q^{-1}} }, 4m+2 \}$.

(b) The group $\pi_1(X^{(r)}_{J(2m+1, -2n)})$ is left-orderable if one of the following holds:

\qquad (i) $n$ is odd and $r>\frac{\pi} { \cos^{-1} \sqrt{ 1-q^{-1}} }$.

\qquad (ii) $n$ is even and $r>\max \{ \frac{\pi} { \cos^{-1} \sqrt{ 1-q^{-1}} }, 4m+2 \}$.
\end{thm}

\begin{remark}
We exclude $J(2m+1,2)$, for $m > 0$, from Theorem \ref{main2} since it is isomorphic to $J(2m,-2)$, and by Remark \ref{r1} the group $\pi_1(X^{(r)}_{J(2m,-2)})$, for $m>0$, is not left-orderable for any integer $r > 1$.
\end{remark}

Here is the plan of the paper. We study non-abelian $SL_2(\BC) $ representations and roots of the Riley polynomial of the knot group of the double twist knots $J(k,l)$ in Section 2. We prove Theorems \ref{main} and \ref{main2} in Section 3.

\section{Non-abelian representations and roots of the Riley polynomial}

\begin{figure}
\setlength{\unitlength}{0.09mm}
\thicklines{
\begin{picture}(300,460)(80,-20)
\put(0,0){\line(0,1){440}}
\put(0,0){\line(1,0){125}}
\put(100,100){\line(1,0){25}}
\put(100,100){\line(0,1){125}}
\put(125,-25){\line(1,0){150}}
\put(125,-25){\line(0,1){150}}
\put(125,125){\line(1,0){150}}
\put(275,-25){\line(0,1){150}}
\put(100,225){\line(1,0){50}}
\put(150,225){\line(0,1){25}}
\put(125,250){\line(1,0){150}}
\put(125,250){\line(0,1){150}}
\put(125,400){\line(1,0){150}}
\put(275,400){\line(0,-1){150}}
\put(0,440){\line(1,0){150}}
\put(150,440){\line(0,-1){40}}
\put(190,30){{\large$l$}}
\put(185,310){\large{$k$}}
\put(250,400){\line(0,1){40}}
\put(250,440){\line(1,0){150}}
\put(400,440){\line(0,-1){440}}
\put(275,0){\line(1,0){125}}
\put(250,225){\line(1,0){50}}
\put(250,225){\line(0,1){25}}
\put(300,100){\line(0,1){125}}
\put(300,100){\line(-1,0){25}}
\end{picture}
}
\caption{The double twist knot $J(k,l)$. Here $k,l$ denote 
the numbers of half twists in each box. Positive numbers correspond 
to right-handed twists and negative numbers correspond to left-handed 
twists. }
\end{figure}
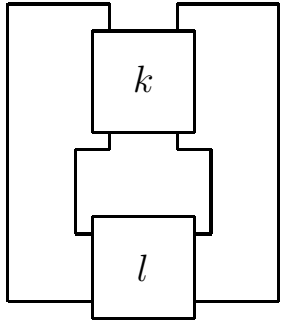

\subsection{Non-abelian representations} By \cite{HSn}, the knot group of $K=J(k,2n)$ is 
$$\pi_1(K) = \la a,b~|~w^na=bw^n \ra,$$ where $a,b$ are meridians and
$$w = 
\begin{cases} 
(ba^{-1})^m(b^{-1}a)^m, & k=2m,\\
(ba^{-1})^mba(b^{-1}a)^m, & k=2m+1.
\end{cases}$$

A representation $\rho: \pi_1(K) \to SL_2(\BC)$ is called non-abelian if 
$\rho(\pi_1(K))$ is a non-abelian subgroup of $SL_2(\BC)$. 
Taking conjugation if necessary, we can assume that $\rho$ 
has the form 
\begin{equation}
\rho(a)=A=\left[ \begin{array}{cc}
s & 1\\
0 & s^{-1} \end{array} \right] \quad \text{and} \quad \rho(b)=B=\left[ \begin{array}{cc}
s & 0\\
2-y & s^{-1} \end{array} \right]
\label{nonabelian}
\end{equation}
where $(s,y) \in \BC^* \times \BC$ satisfies the matrix equation $W^nA-BW^n=0$. Here $W=\rho(w)$.

Let $\{S_j(z)\}_j$ be the sequence of Chebyshev polynomials defined by $S_0(z)=1,\,S_1(z)=z$, and $S_{j+1}(z)=zS_j(z)-S_{j-1}(z)$ for all integers $j$. Note that if $z=t+t^{-1}$, where $t \not= \pm 1$, then $S_{j-1}(t)=\frac{t^j-t^{-j}}{t-t^{-1}}$. Moreover $S_{j-1}(2)=j$ and $S_{j-1}(-2)=(-1)^{j-1}j$ for all integers $j$.

\medskip

The following lemma is standard.

\begin{lemma}
\label{S}
For all integers $j$, one has $$S^2_j(z)-zS_j(z)S_{j-1}(z)+S^2_{j-1}(z)=1.$$
\end{lemma}
Let $x=\tr A =s+s^{-1}$ and $\lambda = \tr W$. The following propositions are proved in \cite{MT}.

\begin{proposition}
\label{p1}
One has
$$
\lambda =
\begin{cases}
2+(y-2)(y+2-x^2)S^2_{m-1}(y), & k=2m,\\
x^2-y-(y-2)(y+2-x^2)S_m(y)S_{m-1}(y), & k=2m+1.
\end{cases}
$$
\end{proposition}

\begin{proposition}
\label{Riley}
One has
$$
W^nA-BW^n = \left[ \begin{array}{cc}
0 & S_{n-1}(\lambda)\alpha-S_{n-2}(\lambda)\\
(y-2) \left( S_{n-1}(\lambda)\alpha-S_{n-2}(\lambda) \right) & 0 \end{array} \right],
$$
where 
$$\alpha = \begin{cases}
1-(y+2-x^2)S_{m-1}(y) \left( S_{m-1}(y) - S_{m-2}(y) \right), & k=2m,\\
1+(y+2-x^2)S_{m-1}(y) \left( S_m(y) - S_{m-1}(y) \right), & k=2m+1.
\end{cases}
$$
\end{proposition}
Proposition \ref{Riley} implies that the assignment \eqref{nonabelian} gives a non-abelian representation 
$\rho: \pi_1(K) \to SL_2(\BC)$ 
%$\rho: \pi_1(K) \to SL_2(\BC)$ 
if and only if $(s,y) \in \BC^* \times \BC$ satisfies the equation 
$$\phi_K(x,y):= S_{n-1}(\lambda)\alpha-S_{n-2}(\lambda)=0.$$
The polynomial $\phi_{K}(x,y)$ is known as the Riley polynomial \cite{Ri} of $K=J(k,2n)$.

\subsection{Roots of the Riley polynomial} In this subsection we prove some properties of the roots of the Riley polynomial of the double twist knots $J(k,l)$.

\begin{lemma}
\label{al}
One has $$\alpha^2-\alpha\lambda+1=
\begin{cases}
(y+2-x^2)S^2_{m-1}(y)(\lambda+2-x^2), & k=2m,\\
(1+(y+2-x^2)S_{m-1}(y)S_m(y))(2-\lambda), & k=2m+1.
\end{cases}$$
\end{lemma}

\begin{proof}
If $k=2m$ then $\alpha=1-(y+2-x^2)S_{m-1}(y) \left( S_{m-1}(y) - S_{m-2}(y) \right)$ and $\lambda=2+(y-2)(y+2-x^2)S^2_{m-1}(y)$. By direct calculations, we have
\begin{eqnarray*}
\alpha^2-\alpha\lambda+1 &=& (y+2-x^2)S^2_{m-1}(y)\\
&&\left[ 2-y+(y+2-x^2) \left( (y-1)S^2_{m-1}(y)-yS_{m-1}(y)S_{m-2}(y)+S^2_{m-2}(y) \right) \right].
\end{eqnarray*}
Since $S^2_{m-1}(y)-yS_{m-1}(y)S_{m-2}(y)+S^2_{m-2}(y)=1$ (by Lemma \ref{S}), we obtain
\begin{eqnarray*}
\alpha^2-\alpha\lambda+1 &=&(y+2-x^2)S^2_{m-1}(y) \left( 4-x^2+(y+2-x^2)(y-2)S^2_{m-1}(y) \right)\\
&=& (y+2-x^2)S^2_{m-1}(y)(\lambda+2-x^2).
\end{eqnarray*}

If $k=2m+1$ then $\alpha=1+(y+2-x^2)S_{m-1}(y) \left( S_m(y) - S_{m-1}(y) \right)$ and $\lambda=x^2-y-(y-2)(y+2-x^2)S_m(y)S_{m-1}(y)$. By direct calculations, we have
\begin{eqnarray*}
\alpha^2-\alpha\lambda+1 &=& (y+2-x^2) \big[ 1-(y+2-x^2)S^2_{m-1}(y)+(2y-x^2)S_{m-1}(y)S_{m}(y)\\
&&+\,(y+2-x^2)S^2_{m-1}(y)\left( S^2_{m-1}(y)-yS_{m-1}(y)S_{m}(y)+(y-1)S^2_{m}(y) \right) \big].
\end{eqnarray*}
Since $S^2_{m-1}(y)-yS_{m-1}(y)S_{m-2}(y)+S^2_{m-2}(y)=1$, we obtain
\begin{eqnarray*}
&& \alpha^2-\alpha\lambda+1\\
 &=&(y+2-x^2) [1+(2y-x^2)S_{m-1}(y)S_{m}(y)+(y+2-x^2)(y-2)S^2_{m-1}(y)S^2_{m}(y)]\\
&=& (y+2-x^2) \left( 1+(y+2-x^2)S_{m-1}(y)S_m(y) \right) \left( 1+(y-2)S_{m-1}(y)S_m(y) \right)\\
&=& (1+(y+2-x^2)S_{m-1}(y)S_m(y))(2-\lambda).
\end{eqnarray*}
This completes the proof of Lemma \ref{al}.
\end{proof}

The following lemma is well known. We include a proof for the reader's convenience.

\begin{lemma}
\label{sin}
For any integer $k$ and any real number $t$, one has $$|\sin kt| \le |k\sin t|.$$
\end{lemma}

\begin{proof}
Without loss of generality we assume that $k \ge 2$. If $k=2m$ ($m \in \BZ_+$) then
$$\sin kt =\sum_{j=1}^m \left( \sin 2jt-\sin (2j-2)t \right)=2\sin t \sum_{j=1}^m \cos (2j-1)t.$$
It follows that $|\sin kt|\le 2m |\sin t|$, since $\left| \sum_{j=1}^m \cos (2j-1)t \right| \le m$.

If $k=2m+1$ ($m \in \BZ_+$) then
$$\sin kt =\sin t + \sum_{j=1}^m \left( \sin (2j+1)t-\sin (2j-1)t \right)=\sin t+2\sin t \sum_{j=1}^m \cos 2jt.$$
It follows that $|\sin kt|\le (2m+1) |\sin t|$, since $\left| 1+2\sum_{j=1}^m \cos 2jt \right| \le 2m+1$.
\end{proof}

\begin{lemma}
\label{S2}
Suppose $z \in \BR$ satisfies $|z| \le 2$. Then $|S_{j-1}(z)| \le |j|$ for all integers $j$.
\end{lemma}

\begin{proof}
If $z=2$ then $S_{j-1}(z)=j$. If $z=-2$ then $S_{j-1}(z)=(-1)^{j-1} j$. 

If $-2<z<2$ we write $z=2\cos t$, where $0<t<\pi$. Then $S_{j-1}(z)=\frac{\sin jt}{\sin t}$, and hence $|S_{j-1}(z)| \le |j|$ by Lemma \ref{sin}.
\end{proof}

\begin{proposition}
Let $K=J(k,2n)$ where $k>0$ and $|n|>0$. Suppose $x, y \in \BR$ satisfy $|x| \le 2$ and $\phi_K(x,y)=0$. Then $y>2$ if one of the following holds:

(a) $k=2m$ ($m \in \BZ_+$) and $|x| > 2 \, \sqrt{1-\frac{1}{|4mn|}}$.

(b) $k=2m+1$ ($m \in \BZ_+$) and $|x| > 2 \, \sqrt{1-\frac{1}{2n^2+2|n| \sqrt{4m(m+1)+n^2}}}$.
\label{y2}
\end{proposition}

\begin{proof}  If $|x|=2$ then by \cite[Prop. 3.2]{MT}, any real root $y$ of $\phi_K(x,y)$ satisfies $y>2$. We now consider the case $|x|<2$.

Suppose $x, y \in \BR$ satisfy $|x| < 2$ and $\phi_K(x,y)=0$. Then $S_{n-1}(\lambda)\alpha - S_{n-2}(\lambda)=0$ and 
\begin{equation}
\label{root}
1 = S^2_{n-1}(\lambda)-\lambda S_{n-1}(\lambda)S_{n-2}(\lambda)+S^2_{n-2}(\lambda) = \left( \alpha^2-\alpha\lambda+1 \right) S^2_{n-1}(\lambda).
\end{equation}

(a) Suppose $k=2m$ ($m \in \BZ_+$) and $|x| >\sqrt{4-\frac{1}{|mn|}}$. By Lemma \ref{al}, $$\alpha^2-\alpha\lambda+1 = (y+2-x^2)S^2_{m-1}(y)(\lambda+2-x^2).$$ 
Eq. \eqref{root} then implies that 
\begin{equation}
\label{=1}
1=(y+2-x^2)S^2_{m-1}(y)(\lambda+2-x^2)S^2_{n-1}(\lambda).
\end{equation}

Assume  $y \le 2$. Since $\lambda - 2=(y-2)(y+2-x^2)S^2_{m-1}(y),$ by Eq. \eqref{=1} we have
$$(\lambda -2)(\lambda+2-x^2)=(y-2)(y+2-x^2)S^2_{m-1}(y)(\lambda+2-x^2)=(y-2)/S^2_{n-1}(\lambda) \le 0$$ which implies that $x^2-2 < \lambda \le 2$. 

Similarly since $(y-2)(y+2-x^2)=(\lambda - 2)/S^2_{m-1}(y) \le 0$, we have $x^2-2 < y \le 2$. 

Since $y \in \BR$ satisfies $|y| \le 2$, $|S_{m-1}(y)| \le |m|$ by Lemma \ref{S2}. Similarly $|S_{n-1}(\lambda)| \le |n|$. Hence, it follows from Eq. \eqref{=1} that
$$1=(y+2-x^2)(\lambda+2-x^2)S^2_{m-1}(y)S^2_{n-1}(\lambda) \le (4-x^2)^2m^2n^2$$
which implies that $x^2 \le 4-\frac{1}{|mn|}$, a contradiction. 

%Assume $y=2$. Then $\lambda =2+(y-2)(y+2-x^2)S^2_{m-1}(y)=2$. Note that $S_j(2)=j+1$ for all integers $j$. We have 
%$\alpha^2-\alpha\lambda+1 = (4-x^2)^2m^2$ and hence $1=(4-x^2)^2m^2n^2$ i.e. $x^2=4-\frac{1}{|mn|}$.

(b) Suppose $k=2m+1$ ($m \in \BZ_+$) and $|x| > \sqrt{4-\frac{2}{n^2+|n| \sqrt{4m(m+1)+n^2}}}$. By Lemma \ref{al}, $$\alpha^2-\alpha\lambda+1 = (1+(y+2-x^2)S_{m-1}(y)S_m(y))(2-\lambda).$$ Eq. \eqref{root} then implies that
\begin{equation}
\label{=1b}
1=(1+(y+2-x^2)S_{m-1}(y)S_m(y))(2-\lambda)S^2_{n-1}(\lambda).
\end{equation}
%\begin{eqnarray*}
%y-2 &=& (y-2)(1+(y+2-x^2)S_{m-1}(y)S_m(y))(2-\lambda) S^2_{n-1}(\lambda)\\
%&=& (\lambda -2)(\lambda+2-x^2) S^2_{n-1}(\lambda).
%\end{eqnarray*}

Assume that $y \le 2$. Since $\lambda + x^2-2  =(2-y)(1+(y+2-x^2)S_{m-1}(y)S_m(y))$, by Eq. \eqref{=1b} we have
\begin{eqnarray*}
(\lambda+2-x^2)(\lambda -2) &=& (2-y)(1+(y+2-x^2)S_{m-1}(y)S_m(y))(\lambda -2)\\
&=& (y-2)/S^2_{n-1}(\lambda) \le 0
\end{eqnarray*} 
which implies that $x^2-2 \le \lambda < 2$. 

Similarly, since $S^2_{m-1}(y)+S^2_m(y)-yS_{m-1}(y)S_m(y)=1$ we have
\begin{eqnarray*}
2-\lambda &=& (y+2-x^2)(1+(y-2)S_{m-1}(y)S_m(y))\\
&=& (y+2-x^2)(S_{m-1}(y)-S_m(y))^2 > 0
\end{eqnarray*} which implies that $y>x^2-2$. Hence, it follows from Eq. \eqref{=1b} that
\begin{eqnarray*}
1 &=& (2-\lambda)S^2_{n-1}(\lambda) \left( 1+(y+2-x^2)S_{m-1}(y)S_m(y) \right)\\
  &\le& (4-x^2)n^2 \left( 1 + (4-x^2)m(m+1)\right)
\end{eqnarray*} 
which implies that $x^2 \le 4-\frac{2}{n^2+|n|\sqrt{4m(m+1)+n^2} }$, a contradiction.
\end{proof}

\begin{proposition}
\label{riley22}
Let $K=J(2m+1,2n)$ where $m \ge 0$ and $n \not= 0, 1, 2$ are integers. Suppose $x \in \BR$ satisfies $|x| \ge 2 \cos \frac{\pi}{4m+2}$. Then the equation $\phi_K(x,y)=0$ has at least one real solution $y>x^2-2$.
\end{proposition}

\begin{proof}
Recall that for $K=J(2m+1,2n)$, $\alpha=1+(y+2-x^2)S_{m-1}(y) \left( S_m(y) - S_{m-1}(y) \right)$ and $\lambda=x^2-y-(y-2)(y+2-x^2)S_m(y)S_{m-1}(y)$. It is obvious that if $y=x^2-2$ then $\alpha=1$ and $\lambda=2$. Hence $$\phi_K(x,x^2-2)=S_{n-1}(\lambda)\alpha - S_{n-2}(\lambda)=S_{n-1}(2) - S_{n-2}(2)=1.$$

We consider the following two cases.

 \textit{\underline{Case 1}: $n \ge 3$.} Note that the polynomial $S_{n-1}(t)-S_{n-2}(t)$ has exactly $n-1$ roots given by $t=2\cos \frac{(2j-1)\pi}{2n-1}$, where $1 \le j \le n-1$. Moreover $$S_{n-1}(2\cos \frac{\pi}{2n-1})>0>S_{n-1}(2\cos \frac{3\pi}{2n-1}).$$
 
Suppose $m=0$. Then $\alpha=1$ and $\lambda=x^2-y$. We have 
$$\phi_K(x, x^2-2\cos \frac{\pi}{2n-1})=S_{n-1}(2\cos \frac{\pi}{2n-1})-S_{n-2}(2\cos \frac{\pi}{2n-1})=0.$$
In this case we choose $y=x^2-2\cos \frac{\pi}{2n-1}$. Then $\phi_K(x,y)=0$ and $y>x^2-2$.

We now suppose $m>0$. Note that $2-\lambda=(y+2-x^2)(S_m(y)-S_{m-1}(y))^2$. Consider the equation $\lambda = 2\cos \frac{3\pi}{2n-1}$, i.e. $(y+2-x^2)(S_m(y)-S_{m-1}(y))^2=2-2\cos \frac{3\pi}{2n-1}.$ It is easy to see that this equation has at least one solution $y_0>x^2-2$. Note that $x^2-2 \ge 2 \cos \frac{\pi}{2m+1}$. Since $y_0>2 \cos \frac{\pi}{2m+1}$, we have $S_m(y_0)>S_{m-1}(y_0)>0$. Hence
\begin{eqnarray*}
\phi_K(x,y_0) &=& S_{n-1}(\lambda)\alpha - S_{n-2}(\lambda)=(\alpha-1) S_{n-1}(\lambda)\\
&=& (y_0+2-x^2)S_{m-1}(y_0) \left( S_m(y_0)-S_{m-1}(y_0) \right) S_{n-1}(2\cos \frac{3\pi}{2n-1})<0.
\end{eqnarray*}
Since $\phi_K(x,x^2-2)>0>\phi_K(x,y_0)$, there exists $y \in (x^2-2,y_0)$ such that $\phi_K(x,y)=0$.

\textit{\underline{Case 2}: $n \le -1$.} Let $l=-n  \ge 1$. We have $$\phi_{K}(x,y):= S_{n-1}(\lambda)\alpha-S_{n-2}(\lambda)=S_l(\lambda)-S_{l-1}(\lambda)\alpha.$$ 

Suppose $m=0$. Then $\alpha=1$ and $\lambda=x^2-y$. In this case we choose $y=x^2-2\cos \frac{\pi}{2l+1}$. Then $\phi_K(x,y)=0$ and $y>x^2-2$.

We now suppose $m>0$. Consider the equation $\lambda = 2\cos \frac{\pi}{2l+1}$, i.e. $(y+2-x^2)(S_m(y)-S_{m-1}(y))^2=2-2\cos \frac{\pi}{2l+1}.$ This equation has at least one real solution $y_0>x^2-2 \ge 2 \cos \frac{\pi}{2m+1}$. We have
\begin{eqnarray*}
\phi_K(x,y_0) &=& S_l(\lambda)-S_{l-1}(\lambda) \alpha \\
&=& -(y_0+2-x^2)S_{m-1}(y_0) \left( S_m(y_0)-S_{m-1}(y_0) \right) S_{l}(2\cos \frac{\pi}{2l+1})<0.
\end{eqnarray*}
Hence there exists $y \in (x^2-2,y_0)$ such that $\phi_K(x,y)=0$.

This completes the proof of Proposition \ref{riley22}.
\end{proof}

\section{Proof of Theorems \ref{main} and \ref{main2}} For a knot $K$ in $S^3$, let $X_K = S^3 \setminus K$ be the knot complement. Recall that
$$ SU(1,1)=\left\{\begin{pmatrix} u & v \\ \bar{v} & \bar{u} \end{pmatrix} : |u|^2 - |v|^2 =1 \right\} \subset SL(2,\mathbb{C})$$
is the special unitary subgroup of $SL(2,\mathbb{C})$. Let $I$ denote the identity matrix in  $SL(2,\mathbb{C})$.

\medskip

The following theorem of Y. Hu is important to us.

\begin{theorem} [\cite{Hu}]
Given any prime knot $K$ in $S^3$, let $\mu$ be a meridian element
of $\pi_1(X_K)$. If there exists a non-abelian representation $\rho : \pi_1(X_K) \to SU(1,1)$ such that
$\rho(\mu^r)= \pm I$ then the fundamental group $\pi_1(X^{(r)}_K)$ is left-orderable.
\label{Hu}
\end{theorem}
\textit{Sketch of the proof of Theorem \ref{Hu}.} Let $\widetilde{SU(1,1)}$ be the universal covering group of $SU(1,1)$. There is a lift of $\rho: \pi_1(X_K) \to SU(1,1)$ to a homomorphism $\widetilde{\rho}: \pi_1(X_K) \to \widetilde{SU(1,1)}$ since the obstruction to its existence is the Euler class $e(\rho) \in H^2(X_K; \BZ) \cong 0$, see \cite{Gh}. Using the Lin's presentation \cite{Li} for the knot group $\pi_1(X_K)$ together with the hypotheses that $\rho(\mu^r)= \pm I$ and $\rho$ is non-abelian, Y. Hu \cite{Hu} shows that the homomorphism $\widetilde{\rho}$ induces a non-trivial homomorphism $\pi_1(X^{(r)}_K) \to \widetilde{SU(1,1)}$. By \cite{BRW, HSt}, a compact, orientable, irreducible 3-manifold has a left-orderable fundamental group
if and only if there exists a non-trivial homomorphism from its fundamental group
to a left-orderable group. We have that $X^{(r)}_K$ is irreducible (since $K$ is prime) and $\widetilde{SU(1,1)}$ is left-orderable. Hence $\pi_1(X^{(r)}_K)$ is left-orderable. This proves Theorem \ref{Hu}.

\medskip

We are ready to prove Theorems \ref{main} and \ref{main2}. For the two-bridge knot $\fb(p,m)$, it is known that the Riley polynomial $\phi_{\fb(p,m)}(x,y)$ is a polynomial in $\BZ[x,y]$ with $y$-leading term $y^d$, where $d=\frac{p-1}{2}$, see \cite{Ri}.

\subsection{Proof of Theorem \ref{main}.} Consider $K=J(2m,2n)$ where $m,n$ are positive integers. Note that $K=\fb(4mn-1, 2n)$ and hence the Riley polynomial $\phi_{K}(x,y)$ is a polynomial in $\BZ[x,y]$ with $y$-leading term $y^d$, where  $d=2mn-1$. Since $d$ is odd, for each $x \in \BR$ the equation $\phi_K(x,y)=0$ has at least one real root $y$. 

For any integer $r > \frac{\pi}{\cos^{-1}\sqrt{1-\frac{1}{4mn}}}$, there is a non-abelian representation $\rho: \pi_1(X_K) \to SL_2(\BC)$ of the form
$$\rho(a)=\left[ \begin{array}{cc}
e^{i\frac{\pi}{r}} & 1\\
0 & e^{-i\frac{\pi}{r}} \end{array} \right] \quad \text{and} \quad \rho(b)=\left[ \begin{array}{cc}
e^{i\frac{\pi}{r}} & 0\\
2-y & e^{-i\frac{\pi}{r}} \end{array} \right]$$
where $y \in \BR$. Note that $x=\tr \rho(a)=2\cos \frac{\pi}{r}$ and $\phi_K(x,y)=0$.

Since $x,y \in \BR$ satisfy $2\sqrt{1-\frac{1}{4mn}}<|x| \le 2$ and $\phi_K(x,y)=0$, Proposition \ref{y2} implies that $y>2$. Since $2-y<0$, a result in \cite{Kh} says that the representation $\rho$ can be conjugated an $SU(1,1)$ representation, denoted by $\rho': \pi_1(X_K) \to SU(1,1)$. Note that $\rho'(a^r)=-I$, since $\rho(a^r)=-I$. Hence Theorem \ref{Hu} implies that $\pi_1(X^{(r)}_K)$ is left-orderable.

\subsection{Proof of Theorem \ref{main2}.} Consider $K=J(2m+1,2n)$ where $m \ge 0$ and $|n|>0$. Note that $K=\fb(4mn+2n-1,2n)$ if $n>0$, and $K=\fb(-4mn-2n+1,-2n)$ if $n<0$. 

Let $q=2n^2 +2|n|\sqrt{4m(m+1)+n^2}$. We consider the following two cases.

\textit{\underline{Case 1}: $n>0$ even or $n<0$ odd.} In this case we have $K=\fb(p,m)$ for some integers $p,m$ such that $p \equiv 3 \pmod{4}$. Hence the Riley polynomial $\phi_{K}(x,y)$ is a polynomial in $\BZ[x,y]$ with $y$-leading term $y^d$, where  $d=\frac{p-1}{2}$ is odd. 

Suppose $r>\frac{\pi} { \cos^{-1} \sqrt{ 1-q^{-1}} }$. Then, by similar arguments as in the proof of Theorem \ref{main}, one can show that the group $\pi_1(X^{(r)}_K)$ is left-orderable. 

\textit{\underline{Case 2}: $n>1$ odd or $n<0$ even.} Suppose $r>\max \{ \frac{\pi} { \cos^{-1} \sqrt{ 1-q^{-1}} }, 4m+2 \}$.

Let $x=2\cos \frac{\pi}{r}$. Since $x\in \BR$ satisfies $|x| \ge 2 \cos \frac{\pi}{4m+2}$, by Proposition \ref{riley22} there exists $y \in \BR$ such that $\phi_K(2\cos \frac{\pi}{r},y)=0$. Hence there is a non-abelian representation $\rho: \pi_1(X_K) \to SL_2(\BC)$ of the form
$$\rho(a)=\left[ \begin{array}{cc}
e^{i\frac{\pi}{r}} & 1\\
0 & e^{-i\frac{\pi}{r}} \end{array} \right] \quad \text{and} \quad \rho(b)=\left[ \begin{array}{cc}
e^{i\frac{\pi}{r}} & 0\\
2-y & e^{-i\frac{\pi}{r}} \end{array} \right].$$ 
Since $x=2\cos \frac{\pi}{r}$ also satisfies $|x|>2\sqrt{1-q^{-1}}$, Proposition \ref{y2} implies that $y>2$. The rest of the proof is similar to that of Theorem \ref{main}.

This completes the proof of Theorem \ref{main2}.


\begin{thebibliography}{99999}

\bibitem[BZ]{BZ} G. Burde, H. Zieschang, {\em Knots}, de Gruyter Stud. Math., vol. 5, de Gruyter, Berlin, 2003.

\bibitem[Be]{Be} G. Bergman, {\em Right orderable groups that are not locally indicable}, Pacific J. Math. \textbf{147} (1991), no. 2, 243--248.

\bibitem[BGW]{BGW} S. Boyer, C. Gordon, and L. Watson, {\em On L-spaces and left-orderable fundamental groups},  Math. Ann. \textbf{356} (2013), no. 4, 1213--1245.

\bibitem[BRW]{BRW} S. Boyer, D. Rolfsen and B. Wiest, {\em Orderable $3$-manifold groups}, Ann. Inst. Fourier (Grenoble) \textbf{55} (2005), no. 1, 243--288.

\bibitem[CLW]{CLW} A. Clay, T. Lidman and L. Watson, {\em Graph manifolds, left-orderability and amalgamation}, Algebr. Geom. Topol. \textbf{13} (2013), no. 4, 2347--2368.

\bibitem[DPT]{DPT} M. Da̧bkowski, J. Przytycki, and A. Togha, {\em Non-left-orderable 3-manifold groups} Canad. Math. Bull. \textbf{48} (2005), no. 1, 32--40.

\bibitem[Gh]{Gh} E. Ghys, {\em Groups acting on the circle}, Enseign. Math. (2) \textbf{47} (2001), no. 3-4, 329--407.

\bibitem[HSn]{HSn} J. Hoste and P. Shanahan, {\em A formula for the A-polynomial of twist knots}, 
J. Knot Theory Ramifications {\bf 14} (2005), 91--100.

\bibitem[HSt]{HSt} J. Howie and H. Short, {\em The band-sum problem}, J. London Math. Soc. (2), \textbf{31} (1985), no. 3, 571--576.

\bibitem[HT1]{HT1} R. Hakamata and M. Teragaito, {\em Left-orderable fundamental group and Dehn surgery on the knot $5_2$}, arXiv: 1208.2087.

\bibitem[HT2]{HT2} R. Hakamata and M. Teragaito, {\em Left-orderable fundamental group and Dehn surgery on twist knots}, arXiv:1212.6305.

\bibitem[HT3]{HT3} R. Hakamata and M. Teragaito, {\em Left-orderable fundamental group and Dehn surgery on genus one two-bridge knots}, preprint 2013, arXiv:1301.2361.

\bibitem[Hu]{Hu} Y. Hu, {\em The left-orderability and the cyclic branched coverings}, preprint 2013, arXiv:1311.3291. 

\bibitem[Kh]{Kh} V. Khoi, {\em A cut-and-paste method for computing the Seifert volumes}, Math. Ann. \textbf{326} (2003), no. 4, 759--801.

\bibitem[Li]{Li} X. Lin, {\em Representations of knot groups and twisted Alexander polynomials}, Acta Math. Sin. (Engl. Ser.) \textbf{17} (2001), no. 3, 361--380.

\bibitem[MT]{MT} T. Morifuji and A. Tran, {\em Twisted Alexander polynomials of 2-bridge knots for parabolic representations}, to appear in Pacific Journal of Mathematics, arXiv:1301.1101. 

\bibitem[OS]{OS} P. Ozsvath and Z. Szabo, {\em On knot Floer homology and lens space surgeries}, Topology \textbf{44} (2005), no. 6, 1281--1300.

\bibitem[Ri]{Ri} R. Riley, {\em Nonabelian representations of 2-bridge knot groups}, Quart. J. Math. Oxford Ser. (2) \textbf{35} (1984), 191--208.

\bibitem[Tr]{Tr} A. Tran, {\em On left-orderable fundamental groups and Dehn surgeries on knots}, to appear in Journal
of the Mathematical Society of Japan, arXiv:1301.2637.

\end{thebibliography}
\end{document}